\documentclass[12pt,oneside,reqno]{amsart}
\usepackage[all]{xy}
\usepackage{amsfonts,amsmath,amssymb,amscd}
\usepackage[breaklinks]{hyperref}
\usepackage{caption} 
\usepackage{graphicx}
\usepackage[utf8]{inputenc}
\renewcommand{\a}{\alpha}
\renewcommand{\b}{\beta}

\numberwithin{equation}{section}

\newtheorem{theorem}{Theorem}

\newtheorem{lemma}{{\bf Lemma}}

\begin{document}
\title[Zeros of combinations of $\Xi(t)$ and ${}_1F_{1}$]{Zeros of combinations of the Riemann $\Xi$-function and the confluent hypergeometric function on bounded vertical shifts} 

\author{Atul Dixit, Rahul Kumar, Bibekananda Maji and Alexandru Zaharescu}

\address{Department of Mathematics, Indian Institute of Technology Gandhinagar, Palaj, Gandhinagar 382355, Gujarat, India} 
\email{adixit@iitgn.ac.in, rahul.kumr@iitgn.ac.in, bibekananda.maji@iitgn.ac.in}

\address{Department of Mathematics, University of Illinois, 1409 West Green
Street, Urbana, IL 61801, USA and \newline
		Simion Stoilow Institute of Mathematics of the Romanian
		Academy, P.O. Box 1--764, RO--014700 Bucharest, Romania.} \email{zaharesc@illinois.edu}

\maketitle
\pagenumbering{arabic}
\pagestyle{headings}

\begin{abstract}
In 1914, Hardy proved that infinitely many non-trivial zeros of the Riemann zeta function lie on the critical line using the transformation formula of the Jacobi theta function. Recently the first author obtained an integral representation involving the Riemann $\Xi$-function and the confluent hypergeometric function linked to the general theta transformation. Using this result, we show that a series consisting of bounded vertical shifts of a product of the Riemann $\Xi$-function and the real part of a confluent hypergeometric function has infinitely many zeros on the critical line, thereby generalizing a previous result due to the first and the last authors along with Roy and Robles. The latter itself is a generalization of Hardy's theorem.

\end{abstract}

\section{Introduction}\label{intro}
Ever since the appearance of Riemann's seminal paper \cite{Rie} in $1859$, the zeros of the Riemann zeta function $\zeta(s)$ have been a constant source of inspiration and motivation for mathematicians to produce beautiful mathematics. While the Riemann Hypothesis, which states that all non-trivial zeros of $\zeta(s)$ lie on the critical line Re$(s)=1/2$, has defied all attempts towards its proof as of yet, the beautiful area of analytic number theory has been blossomed to what it is today because of these attempts.

One of the major breakthroughs in this area occurred with G. H. Hardy \cite{Har} proving that infinitely many non-trivial zeros of $\zeta(s)$ lie on the critical line Re$(s)=1/2$. Let $N_{0}(T)$ denote the number of non-trivial zeros lying on the critical line such that their positive imaginary part is less than or equal to $T$. Hardy and Littlewood \cite{hlmz} showed that $N_{0}(T)>AT$ ,where A is some positive constant. Selberg \cite{selberg} remarkably improved this to $N_{0}(T)>AT\log T$. Levinson \cite{levinson} further improved this by proving that more than one-third of the non-trivial zeros of $\zeta(s)$ lie on the critical line. Conrey \cite{conr} raised this proportion to more than two-fifths. This was later improved by Bui, Conrey and Young \cite{bcy}, Feng \cite{feng}, Robles, Roy and one of the authors \cite{rrz}, with the current record, due to Pratt and Robles \cite{prattrobles}, being that $41.49\%$ of the zeros lie on the critical line.


The proof of Hardy's result in \cite{Har}, which acted as a stimulus to the aforementioned developments, is well-known for its beauty and elegance. One of the crucial ingredients in his proof is the transformation formula satisfied by the Jacobi theta function. The latter is defined by
\begin{align*}
\vartheta(\lambda;\tau) := \sum_{n= -\infty}^{\infty} q^{n^2} \lambda^{n},
\end{align*}
where $q = \exp( \pi i \tau)$ for $\tau \in \mathbb{H}$ (upper half-plane) and $\lambda=e^{2\pi iu}$ for $u \in \mathbb{C}$. If we let $\tau = i x$ for $x>0$ and $u=0$, then the theta function becomes
\begin{align*}
\vartheta(1;i x)= \sum_{n=-\infty}^{\infty} e^{-n^2 \pi x}=:2\psi(x)+1,
\end{align*}
so that $\psi(x)=\sum_{n=1}^{\infty}e^{-n^2\pi x}$. The aforementioned transformation formula employed by Hardy in his proof is due to Jacobi and is given by \cite[p. 22, Equation 2.6.3]{Tit}
\begin{align}\label{Transformation formula for theta function}
\sqrt{x} ( 2\psi(x) + 1) = 2\psi\left( \frac{1}{x} \right) + 1,
\end{align}
which can be alternatively written as
\begin{align}\label{thetaab}
\sqrt{a}\bigg(\frac{1}{2a}-\sum_{n=1}^{\infty}e^{-\pi a^2n^2}\bigg)=\sqrt{b}\bigg(\frac{1}{2b}-\sum_{n=1}^{\infty}e^{-\pi b^2n^2}\bigg)
\end{align}
for Re$(a^2)>0$, Re$(b^2)>0$ and $ab=1$. The former is the same version of the theta transformation using which Riemann \cite{Rie} derived the functional equation of the Riemann zeta function $\zeta(s)$ in the form \cite[p.~22, eqn. (2.6.4)]{Tit}
\begin{equation*}
\pi^{-\frac{s}{2}}\Gamma\left(\frac{s}{2}\right)\zeta(s)=\pi^{-\frac{(1-s)}{2}}\Gamma\left(\frac{1-s}{2}\right)\zeta(1-s).
\end{equation*}
In fact, the functional equation of $\zeta(s)$ is known to be equivalent to the theta transformation. Another crucial step in Hardy's proof is the identity
\begin{align}\label{thetaaxia}
\frac{2}{\pi}\int_{0}^{\infty}\frac{\Xi(t)}{t^2+\frac{1}{4}}\cosh(\a t) {\rm d}t=e^{-\frac{1}{2}i\a}-2e^{\frac{1}{2}i\a}\psi\left(e^{2i\a}\right),
\end{align}
where $\Xi(t)$ is the Riemann $\Xi$-function defined by
\begin{equation*}
\Xi(t)=\xi\left(\tfrac{1}{2}+it\right)
\end{equation*}
with $\xi(s)$ being the Riemann $\xi$-function
\begin{equation*}
\xi(s):=\frac{1}{2}s(s-1)\pi^{-\frac{s}{2}}\Gamma\left(\frac{s}{2}\right)\zeta(s).
\end{equation*}
Equation \eqref{thetaaxia} is easily seen to be equivalent to
\begin{align}\label{thetaaxi}
\frac{2}{\pi}\int_{0}^{\infty}\frac{\Xi(t/2)}{1+t^2}\cos\bigg(\frac{1}{2}t\log a\bigg) {\rm d}t=\sqrt{a}\bigg(\frac{1}{2a}-\sum_{n=1}^{\infty}e^{-\pi a^2n^2}\bigg)
\end{align}
by replacing $t$ by $2t$ and letting $a=e^{i\a}, -\frac{\pi}{4}<\a<\frac{\pi}{4}$, in the latter.

At this juncture, it is important to state that the above theta transformation, that is \eqref{thetaab}, has a generalization \cite[Equation (1.2)]{bgkt}, also due to Jacobi, and is as follows. Let $z \in \mathbb{C}$. If $a$ and $b$ are positive numbers such that $a b = 1$, then
\begin{align}\label{thetaabz}
& \sqrt{a}  \left( \frac{e^{-z^2/8}}{2 a} -  e^{z^2/8} \sum_{n=1}^{\infty} e^{-\pi a^2 n^2} \cos(\sqrt{\pi} a n z ) \right)\nonumber \\
& = \sqrt{b} \left( \frac{e^{z^2/8}}{2 b} -  e^{-z^2/8} \sum_{n=1}^{\infty} e^{-\pi b^2 n^2} \cosh(\sqrt{\pi} b n z ) \right).
\end{align}
It is then natural to seek for an integral representation equal to either sides of the above transformation similar to how the expressions on either sides of \eqref{thetaab} are equal to the integral on the left-hand side of \eqref{thetaaxi}. Such a representation was recently obtained by the first author \cite[Thm. 1.2]{Dix1}. This result is stated below.

\begin{theorem}\label{General Theta Transformation Formula}
Let
\begin{align*}
\nabla(x, z, s) & := \mu(x, z, s) + \mu(x, z, 1-s),\\
\mu(x, z, s) & := x^{1/2 -s} e^{-z^2/8} {}_1 F_1\left(\frac{1-s}{2}; \frac{1}{2}; \frac{z^2}{4} \right),
\end{align*}
where ${}_1F_1(c;d;z):=\sum_{n=0}^{\infty} \frac{(c)_n z^n }{(d)_n n!}$ is the confluent hypergeometric function, $(c)_n := \prod_{i=0}^{n-1} (c+i) = \frac{\Gamma(c+n)}{\Gamma(c)}$, for any $c, d, z,s \in \mathbb{C}$. Then either sides of the transformation in \eqref{thetaabz} equals
\begin{equation}\label{intnabla}
\frac{1}{\pi} \int_{0}^{\infty} \frac{\Xi( \frac{t}{2} )}{1 + t^2} \nabla\left( a, z, \frac{1+it}{2} \right) {\rm d}t.
\end{equation}
\end{theorem}
The next question that comes naturally to our mind is, with the additional parameter $z$ in the above theorem at our disposal, can we generalize the proof of Hardy's result to obtain information on the zeros of a function which generalizes $\Xi(t)$?	Even though one could possibly get a generalization of Hardy's result this way, unfortunately it would not be striking since in the integrand of the generalization of the left side of \eqref{thetaaxia}, one would have $\Xi(t)\left(e^{\a t}{}_1F_{1}\left(\frac{1-2it}{4};\frac{1}{2}; \frac{z^2}{4}\right)+e^{-\a t}{}_1F_{1}\left(\frac{1+2it}{4};\frac{1}{2}; \frac{z^2}{4}\right)\right)$, and then saying that for $z$ fixed in some domain, this function has infinitely many real zeros does not appeal much, for, it is already known that Hardy's theorem implies that $\Xi(t)$ has infinitely many real zeros.

However, even though this approach fails, one can still use the generalized theta transformation in \eqref{thetaabz} and the integral \eqref{intnabla} linked to it to obtain information about zeros of a function which generalizes $\Xi(t)$. This function is obtained through vertical shifts $s\to s+i\tau$ of $\zeta(s)$ and of the confluent hypergeometric function. Many mathematicians have studied the behavior of $\zeta(s)$ on vertical shifts. See \cite{good}, \cite{lirad}, \cite{putnam1}, \cite{putnam2}, \cite{steuweg} and \cite{frank} for some papers in this direction.

Let
\begin{align*}
\eta(s):=\pi^{-s/2}\Gamma\bigg(\frac{s}{2}\bigg)\zeta(s)\quad\text{and}\quad\rho(t):=\eta\bigg(\frac{1}{2}+it\bigg).
\end{align*}
These are related to Riemman's $\xi(s)$ and $\Xi(t)$ functions by $\xi(s)=\tfrac{1}{2}s(s-1)\eta(s)$ and $\Xi(t)=\tfrac{1}{2}\left(\tfrac{1}{2}+it\right)\left(-\tfrac{1}{2}+it\right)\rho(t)$. It is known that $\eta(s)$ is a meromorphic function of $s$ with poles at $0$ and $1$. For real $t$, $\rho(t)$ is a real-valued even function of $t$.

Recently the first and the last authors along with Robles and Roy \cite[Theorem 1.1]{VerticalShift} obtained the following result \footnote{There are some typos in the proof of this result in \cite{VerticalShift}. Throughout the paper, $\lim_{\a\to{\frac{\pi}{4}}^{+}}$ should be changed to $\lim_{\a\to{\frac{\pi}{4}}^{-}}$. The summation sign $\sum_{j=1}^{\infty}$ is missing on the left-hand sides of (3.6) and (3.9). The expression $F\left(\tfrac{1}{2}+i(t+\lambda_{j})\right)$ on the left-hand side of (3.11) should be replaced by $F\left(\tfrac{1}{2}+it\right)$. Also, $m\theta_M$ in (3.12), and $p_n\theta_M$ and $q_n\theta_M$ in (3.16) and (3.17) should be replaced by $2m\theta_M$,  $2p_n\theta_M$ and $2q_n\theta_M$. Finally, the integration on the right side of (3.22) should be performed over $|t|>T$ rather than on $(T, \infty)$.}:

\textit{Let $\{c_j\}$ be a sequence of non-zero real numbers so that $\sum_{j=1}^{\infty} |c_j|<\infty$. Let $\{\lambda_j\}$ be a bounded sequence of distinct real numbers that attains its bounds. Then the function $F(s)=\sum_{j=1}^{\infty}c_j\eta(s+i\lambda_j)$ has infinitely many zeros on the critical line Re$(s)=\frac{1}{2}$.}

Note that the restriction $c_j\neq 0$ is not strict but only for removing redundancy since if $c_j=0$ for some $j\in\mathbb{N}$, then the corresponding term does not contribute anything towards the function $F(s)$. But if we let all but one elements in the sequence $\{c_j\}$ to be equal to zero and the non-zero element, say $c_{j'}$ to be $1$ along with the corresponding $\lambda_{j'}=0$, we recover Hardy's theorem. 

The goal of the present paper is to generalize the above theorem, that is, \cite[Theorem 1.1]{VerticalShift}, which, in turn, as we have seen, generalizes Hardy's theorem. Our main result is as follows.
\begin{theorem}\label{dkmz1main}
 Let $\{c_j \}$ be a sequence of non-zero real numbers so that $\sum_{j=1}^{\infty} |c_j| < \infty$. Let $\{\lambda_j\}$ be a bounded sequence of distinct real numbers such that it attains its bounds. Let $\mathfrak{D}$ denote the region $|\textup{Re}(z) -\textup{Im}(z)| < \sqrt{\frac{\pi}{2}  } - \sqrt{\frac{2}{\pi}}\,\textup{Re}(z) \textup{Im}(z)$ in the $z$-complex plane. Then for any $z\in\mathfrak{D}$, the function $$ F_z(s) := \sum_{j=1}^{\infty}c_j \eta( s + i \lambda_j)  \left\{ {}_1F_1\left( \frac{1-(s+i \lambda_j)}{2};\frac{1}{2}; \frac{z^2}{4} \right) +{}_1F_1\left( \frac{1-( \bar{s} - i \lambda_j)}{2};\frac{1}{2}; \frac{\bar{z}^2}{4} \right) \right\} $$
has infinitely many zeros on the critical line $\textup{Re}(s) = 1/2$. 
\end{theorem}
The region $\mathfrak{D}$ is sketched in Figure 1. The vertices of the square in the center are given by $A=\sqrt{\frac{\pi}{2}}-i\sqrt{\frac{\pi}{2}}, B=\sqrt{\frac{\pi}{2}}+i\sqrt{\frac{\pi}{2}}, C=-\sqrt{\frac{\pi}{2}}+i\sqrt{\frac{\pi}{2}}$ and $D=-\sqrt{\frac{\pi}{2}}-i\sqrt{\frac{\pi}{2}}$.

It is easy to see that when $z=0$, this theorem reduces to the result of the first and the last author along with Robles and Roy \cite[Theorem 1.1]{VerticalShift} given above.

This paper is organized as follows. In Section \ref{prelim}, we collect necessary tools and derive lemmas that are necessary in the proof of Theorem \ref{dkmz1main}. Section \ref{main} is then devoted to proving Theorem \ref{dkmz1main}. In Section \ref{con}, we give concluding remarks and possible directions for future work.
\begin{center}
\begin{figure}[h]\label{fig}
  \includegraphics[totalheight=1.8in]{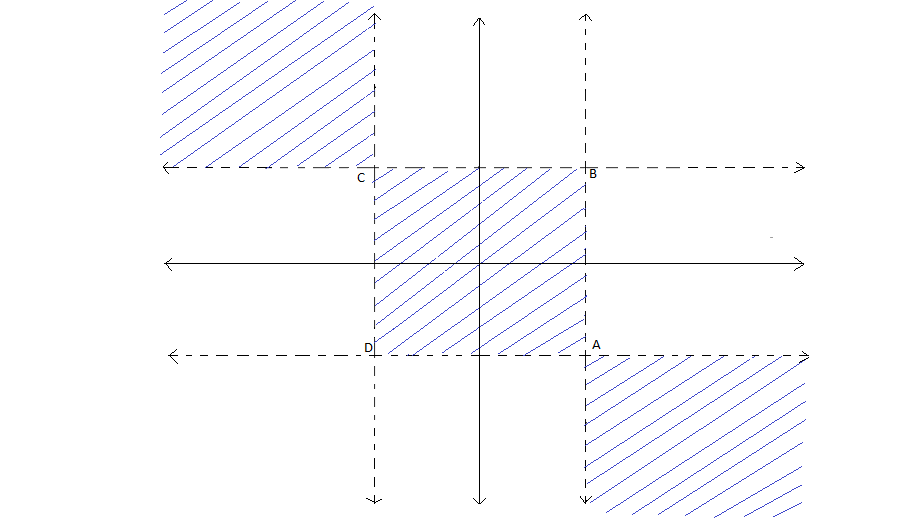}
	\caption{Region $\mathfrak{D}$ in the $z$-complex plane given by \newline $|\textup{Re}(z) -\textup{Im}(z)| < \sqrt{\tfrac{\pi}{2}  } - \sqrt{\tfrac{2}{\pi}}\,\textup{Re}(z) \textup{Im}(z)$.}
\end{figure}
\end{center}
\section{Preliminaries}\label{prelim}
We begin with a lemma which gives a bound on $\Xi(t)$. It readily follows by using elementary bounds on the Riemann zeta function and Stirling's formula on a vertical strip which states that if $s=\sigma+it$, then for $a\leq\sigma\leq b$ and $|t|\geq 1$,
\begin{equation*}
|\Gamma(s)|=(2\pi)^{\tfrac{1}{2}}|t|^{\sigma-\tfrac{1}{2}}e^{-\tfrac{1}{2}\pi |t|}\left(1+O\left(\frac{1}{|t|}\right)\right)
\end{equation*}
as $t\to\infty$.
\begin{lemma}\label{Bound for Riemann Xi function}
For $t \rightarrow \infty$, we have $\Xi(t) = O\left(t^A e^{-\frac{\pi}{4}t} \right)$, where $A$ is some positive constant.
\end{lemma}
We will also need the following estimate for the confluent hypergeometric function proved in \cite[p. 398, Eqn. 4.19]{Dix1}.
\begin{lemma}\label{Bound for 1F1}
For $z \in \mathbb{ C}$ and $|s| \rightarrow \infty$, 
$$
 {}_1F_1\left(-s;\frac{1}{2}; \frac{z^2}{4}\right)= e^{\frac{z^2}{8} } \cos\left(z\sqrt{s + 1/4} \right) + O_{z} \left(\left| s+ 1/4  \right|^{-1/2}   \right).
 $$
\end{lemma}

\begin{lemma}
Let   $\mathfrak{D}$ be a collection of all those complex numbers $z$ such that $|\mathrm{Re}(z) -\mathrm{Im}(z)| < \sqrt{\frac{\pi}{2}  } - \sqrt{\frac{2}{\pi}}\,\mathrm{Re}(z) \mathrm{Im}(z)$. Then 
$
\mathfrak{D} = \Big\{ z : |{\mathrm{Re}(z)}|< \sqrt{\frac{\pi}{2} }, |{\mathrm{Im}(z)}|< \sqrt{\frac{\pi}{2} }\Big\} \bigcup \Big\{z : \mathrm{Re}(z) > \sqrt{\frac{\pi}{2} }, \mathrm{Im}(z) < -\sqrt{\frac{\pi}{2} }   \Big\} \bigcup \Big\{z : \mathrm{Re}(z) <- \sqrt{\frac{\pi}{2} }, \mathrm{Im}(z)  > \sqrt{\frac{\pi}{2} }   \Big\}. 
$  
\end{lemma}

\begin{proof}
Let $c =\sqrt{\frac{\pi}{2} }$, $\mathrm{Re}(z) =x$ and $\mathrm{Im}(z) =y$. One can easily verify that $z\notin\mathfrak{D}$ if $x = \pm c$ or $y=\pm c$. According to the given hypothesis, all points $z=x+iy$ of $\mathfrak{D}$ satisfy
\begin{align*}
| x - y| < c - \frac{x y}{c}.
\end{align*}
We divide the domain into three parts. First, if $ |x | < c $, i.e. $ x+c >0$ and $ x-c <0$, then   
\begin{align}\label{range of y>-c}
x - y < c - \frac{x y}{c} \Leftrightarrow x -c  < \Big( \frac{-y}{c}  \Big) (x -c ) \Leftrightarrow 1 >  \Big( \frac{-y}{c}  \Big)  \Leftrightarrow y > -c. 
\end{align}
Again,
\begin{align}\label{range of y<c}
x - y > -c + \frac{x y}{c} \Leftrightarrow x + c > \frac{y}{c}(x+c) \Leftrightarrow c > y. 
\end{align}
Thus combining \eqref{range of y>-c} and \eqref{range of y<c}, we have $|y| < c$. Therefore, the domain $\Big\{ (x, y) : |x| <c, |y| < c \Big\}$ is a sub-domain of $\mathfrak{D}$. 

Now if $ x > c$, so that $x\pm c>0$. We have 
\begin{align}\label{range for x>c}
x - y < c - \frac{x y}{c} \Leftrightarrow x -c  < \Big( \frac{-y}{c}  \Big) (x -c ) \Leftrightarrow y < -c,
\end{align}
and 
\begin{align}\label{1_range for x>c}
x - y > -c + \frac{x y}{c} \Leftrightarrow x + c > \frac{y}{c}(x+c) \Leftrightarrow  y < c. 
\end{align}
Therefore \eqref{range for x>c} and \eqref{1_range for x>c} imply that $y < -c$ when $x>c$. Similarly it can be seen that if $x<-c$, then $y>c$.
\end{proof}

The following lemma, due to Kronecker, will be used in the proof of the main theorem. See \cite[p.~376, Chapter XXIII]{HW} for proofs.
\begin{lemma}\label{kronecker}
Let $\{n\theta\}$ denote the fractional part of $n\theta$. If $\theta$ is irrational, then the set of points $\{n\theta\}$ is dense in the interval $(0,1)$.
\end{lemma}
Let
\begin{align}\label{Definition of psi(x,z)}
\psi(x, z) := \sum_{n=1}^{\infty} e^{-\pi n^2x} \cos(\sqrt{\pi x} n z).
\end{align}
Replacing $a=\sqrt{x}$ in \eqref{thetaabz}, we have
\begin{align*}
& \sqrt[4]{x} \left( \frac{e^{-z^2/8 }}{2 \sqrt{x} } - e^{z^2/8} \psi(x,z)   \right) = \frac{1}{\sqrt[4]{x}} \left( \frac{e^{z^2/8}}{2} \sqrt{x} - e^{ -z^2/8} \psi\left(\frac{1}{x}, i z \right) \right)
\end{align*}
so that
\begin{align}\label{Transformation formula of psi(x,z)}
& \psi(x,z)  = \frac{e^{-z^2/4 }}{\sqrt{x}} \psi\Big(\frac{1}{x}, i z\Big) + \frac{e^{-z^2/4}}{2\sqrt{x}} - \frac{1}{2}. 
\end{align}
In particular, if we let $z=0$ then we recover \eqref{Transformation formula for theta function}.

We now prove a lemma which is instrumental in the proof of Theorem \ref{dkmz1main}.
\begin{lemma}\label{Convergence of psi function}
Let $z\in\mathfrak{D}$. Then the expressions $\frac{1}{\sqrt{\delta}}e^{-\frac{z^2}{4}\big(1+\frac{i}{\delta}\big)} \psi\left(\frac{1}{4\delta},iz\sqrt{1+\frac{i}{\delta}}\right)$ and 
 $ \frac{1}{\sqrt{\delta}}e^{-\frac{z^2}{4}\big(1+\frac{i}{\delta}\big)}
\psi\left(\frac{1}{\delta},iz\sqrt{1+\frac{i}{\delta}}\right)$ and their derivatives tend to zero as $\delta \rightarrow 0$ along any route in $|\arg( \delta )|< \frac{\pi}{2} $.
\end{lemma}

\begin{proof}
We prove that 
\begin{equation}\label{onelim}
\lim_{\delta\rightarrow 0}\frac{1}{\sqrt{\delta}}e^{-\frac{z^2}{4}\big(1+\frac{i}{\delta}\big)} \psi\left(\frac{1}{4\delta},iz\sqrt{1+\frac{i}{\delta}}\right)=0.
\end{equation}
The second part can be analogously proved. Now
\begin{align}\label{etseries}
\psi\left(\frac{1}{4\delta},iz\sqrt{1+\frac{i}{\delta}}\right)&=\sum_{n=1}^{\infty}e^{-\frac{\pi n^2}{4\delta}}\cos\left(\frac{n\sqrt{\pi}}{2\sqrt{\delta}}iz\sqrt{1+\frac{i}{\delta}}\right)\nonumber\\
&=\sum_{n=1}^{\infty}e^{-\frac{\pi n^2}{4\delta}}\cosh\left(n\sqrt{\pi}z\frac{\sqrt{i+\delta}}{2\delta}\right)\nonumber\\
&=\sum_{n=1}^{\infty}e^{-\frac{\pi n^2}{4\delta}}\left(\frac{e^{n\sqrt{\pi}z\frac{\sqrt{i+\delta}}{2\delta}}+e^{-n\sqrt{\pi}z\frac{\sqrt{i+\delta}}{2\delta}}}{2}\right).
\end{align}
Since the function $\psi(x,z)$ is analytic as a function of $x$ in the right half plane $\mathrm{Re}(x)>0$ and as a function of $z$ for any $z\in\mathbb{C}$, it suffices to show that each term of \eqref{etseries} goes to zero as $\delta$ goes to $0^{+}$.

Therefore,
{\allowdisplaybreaks\begin{align*}
& \lim_{\delta\rightarrow 0^{+}}\left| \frac{1}{\sqrt{\delta}}e^{-\frac{z^2}{4}\big(1+\frac{i}{\delta}\big)} e^{-\frac{\pi n^2}{4\delta}}\left(\frac{e^{n\sqrt{\pi}z\frac{\sqrt{i+\delta}}{2\delta}}+e^{-n\sqrt{\pi}z\frac{\sqrt{i+\delta}}{2\delta}}}{2}\right) \right| \\
 =&  \lim_{\delta\rightarrow 0^{+}} \left| \frac{ e^{-\frac{z^2}{4} }}{2 \sqrt{\delta}} \right| e^{\mathrm{Re}\left(-\frac{z^2 i}{4 \delta}  \right) -\frac{\pi n^2}{4\delta}} \left( e^{\frac{n\sqrt{\pi}}{2\delta} \mathrm{Re}(z\sqrt{i+\delta}) } +  e^{-\frac{n\sqrt{\pi}}{2\delta} \mathrm{Re}(z\sqrt{i+\delta}) } \right)\\
 =&  \lim_{\delta\rightarrow 0^{+}} \left| \frac{ e^{-\frac{z^2}{4} }}{2 \sqrt{\delta}} \right| \left( e^{-\frac{1}{4\delta} \left( \pi n^2 + \mathrm{Re}(z^2 i) - 2 n \sqrt{\pi} \mathrm{Re}(z\sqrt{i+\delta})\right) } + e^{-\frac{1}{4\delta} \left( \pi n^2 + \mathrm{Re}(z^2 i) + 2 n \sqrt{\pi} \mathrm{Re}(z\sqrt{i+\delta})\right) } \right).
\end{align*}}
Note that $\mathrm{Re}\left(z^2 i  \right) = - 2\mathrm{Re}(z)\mathrm{Im}(z) $ and $\mathrm{Re}(z\sqrt{i + \delta}) = \frac{1}{\sqrt{2}}\left(\mathrm{Re}(z) -\mathrm{Im}(z)\right)$, since $\delta \rightarrow 0^{+}$. For $n>1$, the above limit goes to zero because of the presence of $n^2$ in the exponentials. So we will be done if we can show that it goes to zero for $n=1$ too. To ensure this happens, the condition to be imposed on $z$ is
\begin{align*}
& \pi - 2 \mathrm{Re}(z)\mathrm{Im}(z) \pm \sqrt{2 \pi} (\mathrm{Re}(z) -\mathrm{Im}(z)) >0 \\ \nonumber
& \Rightarrow \big| \mathrm{Re}(z) -\mathrm{Im}(z) \big| < \sqrt{\frac{\pi}{2}} - \sqrt{\frac{2}{\pi}}\mathrm{Re}(z)\mathrm{Im}(z),
\end{align*}
The points which satisfy the above inequality are nothing but those which lie in the region $\mathfrak{D}$. This proves \eqref{onelim}. Since $\exp\left(-\frac{1}{\delta} \right) $ goes to zero faster than $\delta^{r}$ for any $r>0$  as $\delta \rightarrow 0^{+}$, derivatives of all orders of $\frac{1}{\sqrt{\delta}}e^{-\frac{z^2}{4}\big(1+\frac{i}{\delta}\big)} \psi\left(\frac{1}{4\delta},iz\sqrt{1+\frac{i}{\delta}}\right)$ also tend to zero as $\delta\to\infty$. This completes the proof of the lemma. 
\end{proof}
Now let 
\begin{equation}\label{psione}
\psi_1(\alpha, z) := e^{\left(\frac{i}{2}-\lambda_j\right)\alpha}\Big(\frac{e^{-z^2/8}}{2}+e^{z^2/8}\psi\big( e^{2i\alpha},z \big) \Big).
\end{equation}
We prove the following result.
\begin{lemma}\label{derivative goes to zero}
Let $z \in \mathfrak{D}$ and $-\frac{\pi}{4} < \alpha < \frac{\pi}{4}$. Then the $2m^{\textup{th}}$ derivative of the function $\psi_1(\alpha, z)$ with respect to $\alpha$, tends to $- \left(\frac{i}{2}-\lambda_j \right)^{2m} e^{\frac{\pi}{4}\left(\frac{i}{2}-\lambda_j\right)} \sinh\left( \frac{z^2}{8} \right) $ as $\alpha \rightarrow {\frac{\pi}{4}}^{-}$. 
\end{lemma}
\begin{proof}
Observe that for $X:=e^{2i\alpha}$, $\lim_{\alpha \rightarrow {\frac{\pi}{4}}^{-}}X=i$.
Note that $X$ tends to $i$ in a circular path where both $x$ and $y$ coordinates are positive. We now let $i + \delta$ tend to  $i$ as $\delta$ goes to zero along any route in $| \arg(\delta) |< \frac{\pi}{2}$.
We first show that $\psi(i + \delta ,z)$  goes to $-1/2$ and its derivatives with respect to $\delta$ go to zero as $\delta \rightarrow 0$ along any route $| \arg(\delta) |< \frac{\pi}{2}$.  
One can easily check that
\begin{align}\label{relation of derivative}\nonumber
\lim_{\alpha \rightarrow {\frac{\pi}{4}}^{-} } \left( \frac{\rm d}{{\rm d}\alpha} \right)^{2m} \psi(e^{2i \alpha}, z) & =
(2 i)^{2m}\lim_{X \rightarrow i} \sum_{j=1}^{2m} a_j X^j \left( \frac{\rm d}{{\rm d}X} \right)^{j}  \psi(X, z) \\
& = (2 i)^{2m}\lim_{\delta \rightarrow 0}\sum_{j=1}^{2m} a_j (i +\delta)^j \left(  \frac{\rm d}{{\rm d}\delta} \right)^{j} \psi(i + \delta, z),
\end{align}
where $a_j$'s are positive integers depending on $j$.

Using the definition \eqref{Definition of psi(x,z)}   of $\psi(x, z)$, we have
	\begin{align*}
		\psi(i+\delta,z)=&\sum_{n=1}^{\infty}e^{-n^2\pi(i+\delta)}\cos\big(\sqrt{\pi(i+\delta)}nz\big)\\
		=&\sum_{n=1}^{\infty}(-1)^ne^{-n^2\pi\delta}\cos\big(\sqrt{\pi(i+\delta)}nz\big)\\
		=&\sum_{n\ even}e^{-n^2\pi\delta}\cos\big(\sqrt{\pi(i+\delta)}nz\big)-\sum_{n\ odd}e^{-n^2\pi\delta}\cos\big(\sqrt{\pi(i+\delta)}nz\big)\\
		=&2\psi\left(4\delta,\frac{\sqrt{i+\delta}}{\sqrt{\delta}}z\right)-\psi\left(\delta,\frac{\sqrt{i+\delta}}{\sqrt{\delta}}z\right). \\
	\end{align*}
Invoking the transformation formula \eqref{Transformation formula of psi(x,z)} of $\psi(x,z)$, we get
	\begin{align*}
		\psi(i+\delta,z)=&\frac{1}{\sqrt{\delta}}e^{-\frac{z^2}{4}\big(1+\frac{i}{\delta}\big)}\psi\left(\frac{1}{4\delta},iz\sqrt{1+\frac{i}{\delta}}\right)-\frac{1}{\sqrt{\delta}}e^{-\frac{z^2}{4}\big(1+\frac{i}{\delta}\big)}\psi\left(\frac{1}{\delta},iz\sqrt{1+\frac{i}{\delta}}\right)-\frac{1}{2}.
	\end{align*}
Along with Lemma \ref{Convergence of psi function}, this implies that
 $\frac{e^{-z^2/8}}{2}+e^{z^2/8}\psi(i+\delta,z)$ tends to $-\sinh \left( \frac{z^2}{8} \right)$ and its derivatives go to zero as $\delta \rightarrow 0$ along any route in an angle $|\arg(\delta)|<\frac{\pi}{2}$, that is, $\frac{e^{-z^2/8}}{2}+e^{z^2/8}\psi(e^{2i \alpha},z)$ goes to $-\sinh\left( \frac{z^2}{8} \right)$, and, due to \eqref{relation of derivative}, its derivatives go to zero as $\alpha \rightarrow {\frac{\pi}{4}}^{-}$. With the help of this result, 
\begin{align*}
&\lim_{\alpha \rightarrow {\frac{\pi}{4}}^{-}} \left( \frac{\rm d}{{\rm d}\alpha} \right)^{2m} \psi_1( \alpha, z) \\
& = \lim_{\alpha \rightarrow {\frac{\pi}{4}}^{-}}\sum_{\substack{ 0 \leq j, k \leq 2m \\ j+k=2m} } {2m \choose j} \left( \frac{\rm d}{{\rm d} \alpha} \right)^j e^{\left(\frac{i}{2}-\lambda_j\right)\alpha}\cdot  \left( \frac{\rm d}{{\rm d} \alpha} \right)^k \left( \frac{e^{-z^2/8}}{2}+e^{z^2/8}\psi(e^{2 i \alpha},z) \right) \\
& = \lim_{\alpha \rightarrow {\frac{\pi}{4}}^{-}}\sum_{\substack{ 0 \leq j<2m, 0<k\leq 2m \\ j+k=2m} } {2m \choose j} \left( \frac{\rm d}{{\rm d} \alpha} \right)^j e^{\left(\frac{i}{2}-\lambda_j\right)\alpha} \cdot \left( \frac{\rm d}{{\rm d} \alpha} \right)^k \left( \frac{e^{-z^2/8}}{2}+e^{z^2/8}\psi(e^{2 i \alpha },z) \right)  \\
& \quad+  {2m\choose 2m} \left( \frac{\rm d}{{\rm d} \alpha} \right)^{2m} e^{\left(\frac{i}{2}-\lambda_j\right)\alpha}  \cdot \left( \frac{\rm d}{{\rm d} \alpha} \right)^0 \left( \frac{e^{-z^2/8}}{2}+e^{z^2/8}\psi(e^{2 i \alpha} ,z) \right)\\
& = - \left( \frac{i}{2}-\lambda_j \right)^{2m} e^{\frac{\pi}{4}\left(\frac{i}{2}-\lambda_j\right)} \sinh \left( \frac{z^2}{8}  \right),
\end{align*}
since the finite sum in the penultimate line vanishes. This completes the proof of the lemma. 
\end{proof}

\section{Proof of Theorem \ref{dkmz1main}}\label{main}
Theorem \ref{General Theta Transformation Formula} implies
\begin{align*}
	\frac{e^{-z^2/8}}{\pi} &\int_{0}^{\infty}\frac{\Xi(\frac{t}{2})}{1+t^2}\left(a^{-it/2}{}_1F_1\Big(\frac{1-it}{4};\frac{1}{2};\frac{z^2}{4}\Big)+a^{it/2}{}_1F_1\Big(\frac{1+it}{4};\frac{1}{2};\frac{z^2}{4}\Big)\right) {\rm{d}}t\\&=\sqrt{a}\left(\frac{e^{-z^2/8}}{2a}-e^{z^2/8}\sum_{n=1}^\infty e^{-\pi a^2n^2}\cos(\sqrt{\pi}a nz)\right).
\end{align*}
Replace $t$ by $2t$ and $a$ by $e^{i\alpha}$, $-\frac{\pi}{4}<\alpha<\frac{\pi}{4}$, and then add and subtract the term $e^{-z^2/8}e^{i\alpha/2}$ on the right hand side of the resulting equation to arrive at
	\begin{align*}
		\frac{e^{-z^2/8}}{\pi} &\int_{0}^{\infty}\frac{\Xi(t)}{\frac{1}{4}+t^2}\left(e^{\alpha t}{}_1F_1\Big(\frac{1-2it}{4};\frac{1}{2};\frac{z^2}{4}\Big)+e^{-\alpha t}{}_1F_1\Big(\frac{1+2it}{4};\frac{1}{2};\frac{z^2}{4}\Big)\right) {\rm{d}}t \nonumber \\  
		 &=2e^{-z^2/8}\cos\alpha/2-2e^{i\alpha/2}\left(\frac{e^{-z^2/8}}{2}+e^{z^2/8}\psi\big(e^{2i\alpha},z\big)\right).
		\end{align*}
The integrand on the left side is an even function of $t$. Hence
{\allowdisplaybreaks\begin{align*}
\frac{e^{-z^2/8}}{\pi} & \int_{- \infty}^{\infty} \rho(t) 	e^{\alpha t}{}_1F_1\Big(\frac{1-2it}{4};\frac{1}{2};\frac{z^2}{4}\Big) {\rm{d}}t \nonumber \\ 
& = -4 e^{-z^2/8}\cos\alpha/2 + 4 e^{i\alpha/2}\left(\frac{e^{-z^2/8}}{2}+e^{z^2/8}\psi\big(e^{2i\alpha},z\big)\right).
	\end{align*}}
Replacing $t$ by $t+\lambda_j$, we find 
	\begin{align}\label{fliden01}
&\int_{-\infty}^{\infty}e^{\alpha t}\rho(t+\lambda_j){}_1F_1\Big(\frac{1-2i(t+\lambda_j)}{4};\frac{1}{2};\frac{z^2}{4}\Big) {\rm{d}}t\nonumber\\
&=\pi e^{-\alpha\lambda_j}\left(-4\cos\alpha/2 + 4e^{z^2/8} e^{i\alpha/2}\left(\frac{e^{-z^2/8}}{2}+e^{z^2/8}\psi\big(e^{2i\alpha},z\big)\right)\right) \nonumber \\
&=-2\pi\bigg[e^{\frac{i\alpha}{2}-\alpha\lambda_j}+e^{-\frac{i\alpha}{2}-\alpha\lambda_j}\nonumber\\
&\quad\quad\quad\quad-2e^{z^2/8}e^{\frac{i\alpha}{2}-\alpha\lambda_j}\bigg(\tfrac{1}{2}e^{-z^2/8}+e^{z^2/8}\sum_{n=0}^{\infty}e^{-n^2\pi e^{2\alpha i}}\cos\left(\sqrt{\pi}ne^{i\alpha}z\right)\bigg)\bigg].
\end{align}
Differentiating both sides $2m$ times with respect to $\alpha$, one gets
\begin{align*}
&\int_{-\infty}^{\infty}t^{2m}e^{\alpha t}\rho(t+\lambda_j){}_1F_1\Big(\frac{1-2i(t+\lambda_j)}{4};\frac{1}{2};\frac{z^2}{4}\Big) {\rm{d}}t\nonumber\\
&=-2\pi\bigg[\bigg(\frac{i}{2}-\lambda_j\bigg)^{2m} e^{\tfrac{\alpha i}{2}-\alpha\lambda_j}+\bigg(\frac{i}{2}+\lambda_j\bigg)^{2m}e^{-\tfrac{\alpha i}{2}-\alpha\lambda_j}\bigg. \nonumber\\
&\hspace{1cm} -2e^{z^2/8} \frac{\partial^{2m}}{\partial \alpha^{2m}}\bigg(e^{\tfrac{\alpha i}{2}-\alpha\lambda_j}\bigg(\frac{e^{-z^2/8}}{2}+e^{z^2/8}\sum_{n=0}^{\infty}e^{-n^2\pi e^{2\alpha i}}\cos\left(\sqrt{\pi}ne^{i\alpha}z\right)\bigg)\bigg)\bigg].
\end{align*}
Let $\frac{i}{2}-\lambda_j=r_je^{i\theta_j}$. Without loss of generality, let $0<\theta_j<\frac{\pi}{2}$. Then
\begin{align}\label{fliden03}
&\int_{-\infty}^{\infty}t^{2m}e^{\alpha t}\rho(t+\lambda_j){}_1F_1\left(\frac{1-2i(t+\lambda_j)}{4};\frac{1}{2};\frac{z^2}{4}\right) {\rm{d}}t\nonumber\\
&=-2\pi e^{-\alpha\lambda_j}\bigg(r_j^{2m} e^{i\left(\tfrac{\alpha}{2}+2m\theta_j\right)}+r_j^{2m} e^{i\left(\tfrac{-\alpha}{2}+2\pi m-2m\theta_j\right)}\bigg)\nonumber\\
&\quad+4\pi e^{z^2/8} \frac{\partial^{2m}}{\partial \alpha^{2m}}\left(e^{\tfrac{\alpha i}{2}-\alpha\lambda_j}\left(\frac{e^{-z^2/8}}{2}+e^{z^2/8}\sum_{n=0}^{\infty}e^{-n^2\pi e^{2\alpha i}}\cos\left(\sqrt{\pi}ne^{i\alpha}z\right)\right)\right)\nonumber\\
&=-4\pi e^{-\alpha\lambda_j}r_j^{2m}\cos \bigg(\frac{\alpha}{2}+2m\theta_j\bigg)+4\pi e^{z^2/8} \frac{\partial^{2m}}{\partial \alpha^{2m}}\psi_1(\alpha,z),
\end{align}	
where $\psi_1(\a, z)$ is defined in \eqref{psione}. Using Lemmas \ref{Bound for Riemann Xi function} and \ref{Bound for 1F1}, we see that
\begin{align*}
 |\rho(t)|  \ll |t|^A e^{-\frac{\pi |t| }{4}}, \quad \mathrm{and} \quad \left| \mathrm{Re}\left({}_1F_{1}\left(\frac{1-2i t}{4};\frac{1}{2};\frac{z^2}{4}\right)\right)  \right| \ll_{z} e^{|z|\sqrt{|t|/2}} 
\end{align*}
as $|t| \rightarrow \infty$, where  $A$ is some positive constant. Since $\{ \lambda_j \}$ is bounded, one sees that 
\begin{align*}
\sum_{j=1}^{\infty} c_j \rho( t + i \lambda_j ) \textup{Re}\left( {}_1F_{1}\left( \frac{1-2i(t+\lambda_j)}{4} ; \frac{1}{2} ; \frac{z^2}{4} \right)\right) \ll_{z} |t|^{A'} e^{- \frac{\pi |t| }{4} + |z|\sqrt{\frac{ |t| }{2} } } \sum_{j=1}^{\infty} | c_{j} |
\end{align*}
as $ |t| \rightarrow \infty$.  Along with the fact that $\sum_{j=1}^{\infty}c_{j}$ converges absolutely, this implies that the above series is uniformly convergent, as a function of $t$, on any compact interval of $(-\infty, \infty)$. 

Take real parts on both sides of \eqref{fliden03}, multiply both sides by $c_j$, sum over $j$, and then interchange the order of summation and integration, which is justified from the uniform convergence of the above series and the fact that $\a<\pi/4$, so as to obtain
\begin{align}\label{fliden04}
&\int_{-\infty}^{\infty}t^{2m} e^{\alpha t}\sum_{j=1}^{\infty}c_j\rho(t+i\lambda_j)\textup{Re}\left({}_1F_{1}\left(\frac{1-2i(t+\lambda_j)}{4};\frac{1}{2};\frac{z^2}{4}\right)\right) {\rm {d}}t\nonumber\\
&=-4\pi\sum_{j=1}^{\infty}c_j e^{-\alpha\lambda_j}r_j^{2m}\cos \bigg(\frac{\alpha}{2}+2m\theta_j\bigg)+4\pi\textup{Re}\left[e^{\frac{z^2}{8}}\sum_{j=1}^{\infty}c_j\frac{\partial^{2m}}{\partial \alpha^{2m}}\psi_1(\alpha,z)\right].
\end{align}
Now using the notation $F_{z}\left(\frac{1}{2}+it\right)$ for the series on the left-hand side of \eqref{fliden04} as defined in the statement of Theorem \ref{dkmz1main} and letting $\alpha\to{\frac{\pi}{4}}^{-}$ on both sides, we see that
\begin{align}\label{fliden05}
&\lim_{\alpha\to{\frac{\pi}{4}}^{-}}\int_{-\infty}^{\infty}t^{2m} e^{\alpha t}F_{z}\left(\frac{1}{2}+it\right) {\rm {d}}t\nonumber\\
&=-4\pi\sum_{j=1}^{\infty}c_j e^{-\frac{\pi}{4}\lambda_j}r_j^{2m}\cos \bigg(\frac{\pi}{8}+2m\theta_j\bigg)\nonumber\\
&\quad\quad-4\pi\textup{Re}\left[e^{\frac{z^2}{8}}\sum_{j=1}^{\infty}c_j \left( \frac{i}{2}-\lambda_j \right)^{2m} e^{\frac{\pi}{4}\left(\frac{i}{2}-\lambda_j\right)} \sinh \left( \frac{z^2}{8}  \right)\right]\nonumber\\
&=-4\pi\sum_{j=1}^{\infty}c_j e^{-\frac{\pi}{4}\lambda_j}r_j^{2m}\left\{\cos \bigg(\frac{\pi}{8}+2m\theta_j\bigg)+\textup{Re}\left[e^{i\left(\frac{\pi}{8}+2m\theta_j\right)}e^{\frac{z^2}{8}}\sinh \left( \frac{z^2}{8}  \right)\right]\right\},
\end{align}
where in the penultimate step we used Lemma \ref{derivative goes to zero}.

Note that if $z$ is real or purely imaginary, the right-hand side of \eqref{fliden05} becomes
\begin{align*}
\left(1+e^{\frac{z^2}{8}}\sinh \left( \frac{z^2}{8}  \right)\right)\left\{-4\pi\sum_{j=1}^{\infty}c_j e^{-\frac{\pi}{4}\lambda_j}r_j^{2m}\cos \bigg(\frac{\pi}{8}+2m\theta_j\bigg)\right\},
\end{align*}
and thus the logic to prove that the above expression changes sign infinitely often is similar to that in \cite{VerticalShift}.

Now let us assume that $z$ is a complex number lying in the region $\mathfrak{D}$ that is neither real nor purely imaginary. Then from \eqref{fliden05},
\begin{align}\label{new}
&\lim_{\alpha\to{\frac{\pi}{4}}^{-}}\int_{-\infty}^{\infty}t^{2m} e^{\alpha t}F_{z}\left(\frac{1}{2}+it\right) {\rm{d}}t\nonumber\\
&=-4\pi\sum_{j=1}^{\infty}c_j e^{-\frac{\pi}{4}\lambda_j}r_j^{2m}\bigg\{\cos \bigg(\frac{\pi}{8}+2m\theta_j\bigg)\left(1+\textup{Re}\left(e^{\frac{z^2}{8}}\sinh \left( \frac{z^2}{8}  \right)\right)\right)\nonumber\\
&\qquad\qquad\qquad\qquad\qquad\quad-\sin\bigg(\frac{\pi}{8}+2m\theta_j\bigg)\textup{Im}\left(e^{\frac{z^2}{8}}\sinh \left( \frac{z^2}{8}  \right)\right)\bigg\}.
\end{align}
Now let 
\begin{align}\label{uv}
u_z:=1+\textup{Re}\left(e^{\frac{z^2}{8}}\sinh \left( \frac{z^2}{8}  \right)\right)\hspace{2mm}\text{and} \hspace{2mm}v_z:=\textup{Im}\left(e^{\frac{z^2}{8}}\sinh \left( \frac{z^2}{8}  \right)\right).
\end{align}
From \eqref{new} and \eqref{uv},
\begin{align}\label{new1}
&\lim_{\alpha\to{\frac{\pi}{4}}^{-}}\int_{-\infty}^{\infty}t^{2m} e^{\alpha t}F_{z}\left(\frac{1}{2}+it\right) {\rm{d}}t=-4\pi w_z\sum_{j=1}^{\infty}c_j e^{-\frac{\pi}{4}\lambda_j}r_j^{2m}\cos \bigg(\frac{\pi}{8}+\b_z+2m\theta_j\bigg),
\end{align}
where
\begin{align*}
w_z:=\sqrt{u_z^2+v_z^2}\hspace{2mm}\textup{and}\hspace{2mm}\b_z:=\cos^{-1}\left(\frac{u_z}{w_z}\right).
\end{align*}
Since $u_z$ and $v_z$ are real, the quantities $w_z$ and $\b_z$ are real too. We now show that the series on the right side of \eqref{new1} changes sign infinitely often.

By the hypothesis, there exists a positive integer $M$ such that 
\begin{align*}
|\lambda_{M}|=\max_{j}\{|\lambda_j|\}, u\quad\text{and}\quad \lambda_M\neq\lambda_j\quad\text{for}\quad M\neq j.
\end{align*}
Then the series on the right-hand side of \eqref{new1} (without the constant term in the front) can be written as 
\begin{align}\label{rightside}
c_M r_M^{2m}e^{-\frac{\pi \lambda_M}{4}}\cos\bigg(\frac{\pi}{8}+\b_z+2m\theta_M\bigg)(1+E(X, z)+ H(X, z)),
\end{align}
where 
\begin{align}\label{ex}
E(X, z):=\sum_{\substack{j\neq M\\j\leq X}}\frac{c_j}{c_M}e^{-\frac{\pi}{4}(\lambda_j-\lambda_M)}\bigg(\frac{r_j}{r_M}\bigg)^{2m}\frac{\cos (\frac{\pi}{8}+\b_z+2m\theta_j)}{\cos (\frac{\pi}{8}+\b_z+2m\theta_M)},
\end{align}
as well as
\begin{align}\label{hx}
H(X, z):=\sum_{\substack{j\neq M\\j> X}}\frac{c_j}{c_M}e^{-\frac{\pi}{4}(\lambda_j-\lambda_M)}\bigg(\frac{r_j}{r_M}\bigg)^{2m}\frac{\cos (\frac{\pi}{8}+\b_z+{2m}\theta_j)}{\cos (\frac{\pi}{8}+\b_z+2m\theta_M)},
\end{align}
$X$ being a real number that is sufficiently large.

We now claim that there exists a subsequence of natural numbers such that for each value $m$ in it, the inequality $\lvert\cos (\frac{\pi}{8}+\beta_z+2m\theta_M)\rvert\geq c$ holds for some positive constant $c$.

\noindent
Note that $\frac{i}{2}-\lambda_M=r_Me^{i\theta_M}$ for $0<\theta_M<\frac{\pi}{2}$. Then 
\begin{align}\label{rmrj}
r_M>r_j\quad\text{for}\quad M\neq j.
\end{align}
Observe that
\begin{align*}
\cos \left(\frac{\pi}{8}+\b_z+2m\theta_M\right)&=\cos \left(\frac{\pi}{8}+\b_z+2\pi\left(\frac{m\theta_M}{\pi}\right)\right)\nonumber\\
&=\cos \left(\frac{\pi}{8}+\b_z+2\pi\left\lfloor\frac{m\theta_M}{\pi}\right\rfloor+2\pi\left\{\frac{m\theta_M}{\pi}\right\}\right)\nonumber\\
&=\cos \left(\frac{\pi}{8}+\b_z+2\pi\left\{\frac{m\theta_M}{\pi}\right\}\right).
\end{align*}
In the remainder of the proof we construct two subsequences $\{p_n\}$ and $\{q_n\}$ of $\mathbb{N}$ such that $\{\frac{p_n\theta_M}{\pi}\}$ and $\{\frac{q_n\theta_M}{\pi}\}$ tend to some specific numbers inside the interval $(0,1)$ resulting in 
$$ \cos \left(\frac{\pi}{8}+\b_z+2\pi\left\{\frac{m\theta_M}{\pi}\right\}\right)>0 \quad \mathrm{and} \quad \cos \left(\frac{\pi}{8}+\b_z+2\pi\left\{\frac{m\theta_M}{\pi}\right\}\right)<0,$$
where for the first inequality, $m$ takes values from the sequence $\{p_n\}$ with $n\geq N$ where $N$ is large enough, and for the second, $m$ takes values from $\{q_n\}$ with $n\geq N$.



To that end, we divide the proof of the claim into two cases. First consider the case when $\frac{\theta_M}{\pi}$ is irrational. This case itself is divided into five subcases depending upon the location of $\beta_z$ in the interval $[0,2\pi]$. In all these cases, Kronecker's lemma, that is, Lemma \ref{kronecker} plays an instrumental role.

\textbf{Case 1:} Let $0\leq \b_z<\frac{3}{8}\pi$.

Take $j$ to be a large enough natural number and consider the subsequence $\{p_n\}$ of $\mathbb{N}$ such that $\left\{\frac{p_n\theta_M}{\pi}\right\}\to\frac{1}{2^{j+1}}$ and so that $\frac{\pi}{8}<\lim_{n\to\infty}\left(\frac{\pi}{8}+\b_z+2\pi\left\{\frac{p_n\theta_M}{\pi}\right\}\right)<\frac{\pi}{2}$. This ensures that $\cos \left(\frac{\pi}{8}+\b_z+2\pi\left\{\frac{m\theta_M}{\pi}\right\}\right)>0$ for $n\geq N$ for some $N\in\mathbb{N}$ large enough. 

It is also clear that if we let $m$ run through the subsequence $\{q_n\}$ of $\mathbb{N}$ such that $\left\{\frac{q_n\theta_M}{\pi}\right\}\to\frac{1}{4}$, then $\cos \left(\frac{\pi}{8}+\b_z+2\pi\left\{\frac{m\theta_M}{\pi}\right\}\right)<0$ for $n\geq N$ for some $N\in\mathbb{N}$ large enough. 
	
In the subcases that follow, the argument is similar to that in Case 1, and hence in each such case we only give the two subsequences $\{p_n\}$ and $\{q_n\}$ that we can let $m$ run through so that $\cos \left(\frac{\pi}{8}+\b_z+2\pi\left\{\frac{m\theta_M}{\pi}\right\}\right)>0$ and $\cos \left(\frac{\pi}{8}+\b_z+2\pi\left\{\frac{m\theta_M}{\pi}\right\}\right)<0$ respectively for $n\geq N$ for some $N\in\mathbb{N}$ large enough. 

\textbf{Case 2:} Let $\frac{3\pi}{8}\leq \b_z<\frac{7}{8}\pi$.

Choose $\{p_n\}$ to be such that $\left\{\frac{p_n\theta_M}{\pi}\right\}\to\frac{3}{4}$ and $\{q_n\}$ to be such that $\left\{\frac{q_n\theta_M}{\pi}\right\}\to\frac{1}{8}$.

\textbf{Case 3:} Let $\frac{7\pi}{8}\leq \b_z<\frac{11}{8}\pi$.

Here we can select $\{p_n\}$ so that $\left\{\frac{p_n\theta_M}{\pi}\right\}\to\frac{3}{8}$ and $\{q_n\}$ so that $\left\{\frac{q_n\theta_M}{\pi}\right\}\to\frac{1}{2^{j'+1}}$, where $j'\in\mathbb{N}$ is large enough so that $\pi<\lim_{n\to\infty}\left(\frac{\pi}{8}+\b_z+2\pi\left\{\frac{q_n\theta_M}{\pi}\right\}\right)<\frac{3\pi}{2}$.

\textbf{Case 4:} Let $\frac{11\pi}{8}\leq \b_z<\frac{15}{8}\pi$.

Choose $\{p_n\}$ to be such that $\left\{\frac{p_n\theta_M}{\pi}\right\}\to\frac{1}{8}$ and $\{q_n\}$ to be such that $\left\{\frac{q_n\theta_M}{\pi}\right\}\to\frac{3}{4}$.

\textbf{Case 5:} Let $\frac{15\pi}{8}\leq \b_z<2\pi$.

Here we can allow $\{p_n\}$ to be such that $\left\{\frac{p_n\theta_M}{\pi}\right\}\to\frac{1}{16}$ and $\{q_n\}$ to be such that $\left\{\frac{q_n\theta_M}{\pi}\right\}\to\frac{1}{2}$.

From the above construction it is clear that, according to the location of $\beta_z$, we can always find a positive real number $c$ such that   
\begin{align}\label{BoundOfcos}
\left\lvert\cos \left(\frac{\pi}{8}+\beta_z+2m\theta_M\right)\right\rvert\geq c,
\end{align}
when $m$ runs over the sequence $\{p_n\} \cup \{q_n \}$ for $n\geq N$, where $N$ is large enough. 

If $m$ runs over the above mentioned sequence then \eqref{hx}, \eqref{rmrj} and \eqref{BoundOfcos} imply
\begin{align}\label{hx1}
\left|H(X,z)\right|&\leq \frac{1}{c|c_M|}\sum_{\substack{ j\neq M \\ j>X}} |c_j|\ e^{-\frac{\pi}{4}(\lambda_j-\lambda_M)}.
\end{align}

By our hypothesis $\{ \lambda_j-\lambda_M \} $ is also a bounded sequence, so that $ e^{-\frac{\pi}{4}(\lambda_j-\lambda_M)}<A_1$ for some positive constant $A_1$. Therefore from \eqref{hx1},
\begin{align}\label{hxbound}
\left|H(X,z)\right|&\leq \frac{A_1}{c|c_M|}\sum_{\substack{ j\neq M \\ j>X}} |c_j|.
\end{align}
Since $\sum_{j=1}^{\infty} |c_j|$ is convergent, this implies $|H(X,z)|=O(1)$ for $X$ large enough.

Now $C_X:=\max_{j\leq X}\left\{\frac{|r_j|}{|r_M|}\right\}$ is finite, in fact, \eqref{rmrj} implies $C_X<1$. Similarly using \eqref{ex} and \eqref{rmrj}, it can be shown that when $m$ runs over the same sequence,
\begin{align}\label{exbound}
|E(X,z)|&\leq\frac{A_2\ C_X^{2m}}{c|c_M|}\sum_{\substack{j\neq M\\ j\leq X}} |c_j|,
\end{align}%
where $A_2$ is some constant, independent of $m$. Since $C_X<1$, we conclude that $ E(X,z)\to 0$ as $m\to\infty$ through the above sequence $\{p_n \} \cup \{ q_n \}$. 
 
It is to be noted that when $m$ runs over the sequence that we have constructed, $\cos\left(\frac{\pi}{8}+\beta_z+2m\theta_M\right)$ changes sign infinitely often. 
Thus, from \eqref{rightside}, \eqref{hxbound} and \eqref{exbound}, it is clear that the right hand side of \eqref{new1} changes sign infinitely often for infinitely many values of $m$.

Our aim is to prove that the function $F_z(s)$ has infinitely many zeros on the critical line Re$(s) = 1/2$.
Suppose not. Then $F_z\left(\frac{1}{2}+it\right)$ never changes sign for $|t|>T$ for some $T$ large. First, consider $F_z\left(\frac{1}{2}+it\right) > 0 $ for $|t|>T$. Define
\begin{align*}
L_{z, m}(T):=\lim_{\alpha\rightarrow\frac{\pi}{4}^-}\int_{|t|\geq T}F_z\left(\frac{1}{2}+it\right)t^{2m}e^{\alpha t}\ \mathrm{d}t.
\end{align*}
Since the integrand  of the above integral  is positive, for any $T_1 > T$,  
\begin{align*}
\lim_{\alpha\rightarrow\frac{\pi}{4}^-}\int_{ T \leq |t| \leq T_1} F_z\left(\frac{1}{2}+it\right)t^{2m}e^{\alpha t}\ \mathrm{d}t \leq \lim_{\alpha\rightarrow\frac{\pi}{4}^-}\int_{|t|\geq T}F_z\left(\frac{1}{2}+it\right)t^{2m}e^{\alpha t}\ \mathrm{d}t =L_{z, m}(T) . 
\end{align*}
Therefore, 
\begin{align*}
\int_{ T \leq |t| \leq T_1} F_z\left(\frac{1}{2}+it\right)t^{2m}e^{\frac{\pi}{4} t}\ \mathrm{d}t \leq  L_{z, m}(T).
\end{align*}
Now if $T_1$ tends to $\infty$,
\begin{align*}
& \int_{ T \leq |t| } F_z\left(\frac{1}{2}+ i t\right)t^{2m}e^{\frac{\pi}{4} t}\ \mathrm{d}t \leq  L_{z, m}(T)
\end{align*}
so that
\begin{align*}
\int_{-\infty}^{\infty} F_z\left(\frac{1}{2}+ i t\right)t^{2m}e^{\frac{\pi}{4} t}\ \mathrm{d}t \leq  L_{z, m}(T) + \int_{- T }^{ T } F_z\left(\frac{1}{2}+ i t\right)t^{2m}e^{\frac{\pi}{4} t}\ \mathrm{d}t.
\end{align*}
Since the integrand on the right hand side is an analytic function of $t$ in $[-T, T]$, $\int_{-\infty}^{\infty} F_z\left(\frac{1}{2}+ i t\right)t^{2m}e^{\frac{\pi}{4} t}\ \mathrm{d}t$ is convergent. Using \cite[p.~149, Theorem 7.11]{rudin}, for example, it can be checked that that the integrand on the left hand side of \eqref{new1} is uniformly convergent, with respect to $\alpha$, on $ 0 \leq \alpha < \frac{\pi}{4}$. Then \eqref{new1} implies
\begin{align*}
\int_{-\infty}^{\infty}t^{2m} e^{\frac{\pi}{4} t}F_{z}\left(\frac{1}{2}+it\right)\, \mathrm{d}t=-4\pi w_z\sum_{j=1}^{\infty}c_j e^{-\frac{\pi}{4}\lambda_j}r_j^{2m}\cos \bigg(\frac{\pi}{8}+\b_z+2m\theta_j\bigg),
\end{align*}
for every $m \in \mathbb{N}$. As per our construction of the sequences $\{p_n\}$ and $\{q_n\}$, there exist infinitely many $m \in \{p_n\} \cup \{ q_n \} $ large enough such that the right hand side of the above identity is negative, and hence 
\begin{align}\label{upperBound}
 \int_{ T \leq |t| } F_z\left(\frac{1}{2}+ i t\right)t^{2m}e^{\frac{\pi}{4} t}\ \mathrm{d}t & < - \int_{ |t| \leq T } F_z\left(\frac{1}{2}+ i t\right)t^{2m}e^{\frac{\pi}{4} t} \mathrm{d}t \nonumber \\
& < T^{2 m} \int_{ |t| \leq T }\left| F_z\left(\frac{1}{2}+ i t\right) e^{\frac{\pi}{4} t}\right|  \mathrm{d}t \nonumber \\
& \leq B T^{2m},
\end{align}
where $B:=B(T)$ is independent of $m$. 

By our assumption on $F_z\left(\frac{1}{2}+it\right)$, we can find a positive number $\delta = \delta(T)$ such that $F_z\left(\frac{1}{2}+it\right) > \delta $ for all $t \in (2T , 2T + 1) $ . Hence
\begin{align}\label{lowerBound}
\int_{ T \leq |t| } F_z\left(\frac{1}{2}+ i t\right)t^{2m}e^{\frac{\pi}{4} t}\ \mathrm{d}t & \geq \int_{2T}^{2T+1} \delta t^{2m} e^{\frac{\pi}{4}t}\mathrm{d}t \nonumber\\
&\geq \delta \int_{2T}^{2T+1} t^{2m} \mathrm{d}t \nonumber \\
&=\delta\left(\frac{(2T+1)^{2m+1}}{2m+1}-\frac{(2T)^{2m+1}}{2m+1}\right) \nonumber \\
&\geq \delta(2T)^{2m}.
\end{align}
From \eqref{upperBound} and \eqref{lowerBound},
\begin{align*}
\delta(2T)^{2m} \leq \int_{ T \leq |t| } F_z\left(\frac{1}{2}+ i t\right)t^{2m}e^{\frac{\pi}{4} t}\ \mathrm{d}t < B T^{2m} 
\end{align*}
for infinitely many large $m\in\{p_n\}\cup\{q_n\}.$
This implies that 
\begin{align}\label{contradiction}
 2^{2m}\delta< B
\end{align}
holds for infinitely many $m\in\{p_n\}\cup\{q_n\}.$ However this is impossible since $m$ can be chosen to be arbitrarily large. So our assumption that $F_z\left(\frac{1}{2}+it\right) > 0 $ for $|t|>T$ is not true.


Similar contradiction can be reached at when $F_z\left(\frac{1}{2}+it\right) < 0 $ for $|t|>T$ and $\frac{\theta_M}{\pi}$ is irrational.

Lastly if $F_z\left(\frac{1}{2}+it\right)>0$ for $t>T$ and $F_z\left(\frac{1}{2}+it\right)<0$ for $t<-T$ (or vice-versa), we differentiate \eqref{fliden01} $2m+1$ times with respect to $\alpha$. If $\frac{\theta_M}{\pi}$ is irrational, one can construct the subsequences $\{p_n\}$ and $\{q_n\}$ of $\mathbb{N}$, similarly as done in the first case, that is, according to the location of $\beta_z$.

For example, if $0\leq \beta_z<\frac{3\pi}{8}$, one can find $j\in\mathbb{N}$ large enough such that $\frac{\pi}{8}+\beta_z+\frac{\pi}{2^j}<\frac{\pi}{2}$. Now if $\frac{\theta_M}{2\pi}<\frac{1}{2^{j+1}}$, then by Kronecker's lemma (Lemma \ref{kronecker}), there exists a sequence $\{p_n\}$ such that $\{\frac{p_n\theta_M}{\pi}\}\rightarrow \frac{1}{2^{j+1}}-\frac{\theta_M}{2\pi}$, so that 
$\cos \left(\frac{\pi}{8}+\b_z+(2p_n+1)\theta_M\right)>0$. Now if $\frac{\theta_M}{2\pi}>\frac{1}{2^{j+1}}$, by Kronecker's lemma again one can find a sequence $\{p_n\}$ such that $\{\frac{p_n\theta_M}{\pi}\}-1\rightarrow \frac{1}{2^{j+1}}-\frac{\theta_M}{2\pi}$. Since
\begin{align*}
\frac{\pi}{8}+\b_z+2\pi(2p_n+1)\frac{\theta_M}{2\pi}
=\frac{\pi}{8}+\b_z+2\pi\left(\left\lfloor\frac{p_n\theta_M}{\pi}\right\rfloor+1+\left\{\frac{p_n\theta_M}{\pi}\right\}-1+\frac{\theta_M}{2\pi}\right),
\end{align*}
by periodicity,
\begin{equation*}
\cos\left(\frac{\pi}{8}+\b_z+2\pi(2p_n+1)\frac{\theta_M}{2\pi}\right)=\cos\left(\frac{\pi}{8}+\b_z+2\pi\left(\left\{\frac{p_n\theta_M}{\pi}\right\}-1+\frac{\theta_M}{2\pi}\right)\right).
\end{equation*}
Thus,
\begin{align*} 
\cos\left(\frac{\pi}{8}+\b_z+2\pi\left(\left\{\frac{p_n\theta_M}{\pi}\right\}-1+\frac{\theta_M}{2\pi}\right)\right) = \cos \left( \frac{\pi}{8}+\b_z+ \frac{2\pi}{2^{ j+1}} \right) >0
\end{align*} 
for $n\geq N$, where $N\in\mathbb{N}$ is large enough.

Similarly we can find a subsequence $\{q_n\}$ of $\mathbb{N}$ such that $\left\{\frac{q_n\theta_M}{\pi}\right\}\to\frac{1}{4} - \frac{\theta_M }{2 \pi}$, then $\cos \left(\frac{\pi}{8}+\b_z+2\pi\left\{\frac{q_n\theta_M}{\pi}\right\}\right)<0$ for $n\geq N$, where $N\in\mathbb{N}$ is large enough.

Similarly corresponding to other locations of $\b_z$, one can obtain corresponding subsequences $\{p_n\}$ and $\{q_n\}$ of $\mathbb{N}$ such that  $\cos\left(\frac{\pi}{8}+\b_z+(2p_n+1)\theta_M\right)>0$ and that $\cos\left(\frac{\pi}{8}+\b_z+(2q_n+1)\theta_M\right)<0$ for $n\geq N$, where $N\in\mathbb{N}$ is large enough. One can then use them to obtain a contradiction similar to that in \eqref{contradiction} by using argument similar to that given in equations \eqref{BoundOfcos} through \eqref{contradiction}. 

Thus, we conclude that $F_{z}\left(\tfrac{1}{2}+it\right)$ changes sign infinitely often. This proves Theorem \ref{dkmz1main} when $\theta_M/\pi$ is irrational.

It only remains to prove Theorem \ref{dkmz1main} in the case when $\theta_M/\pi$ is rational. We reduce it to the previous case, that is, when $\theta_M/\pi$ is irrational, by performing the following trick. Fix a small positive real number $\epsilon_0$ and consider
the function $F_{\epsilon_0,z}(s) := F_z(s+i\epsilon_0)$.
Then $F_{\epsilon_0,z}$ is a vertical shift of $F_z$,
and the statement that $F_z$ has infinitely many
zeros on the critical line is equivalent to the
statement that $F_{\epsilon_0,z}$ has infinitely many
zeros on the critical line. Moreover, $F_{\epsilon_0,z}(s)$
satisfies the conditions from the statement 
of Theorem 2, with the same value of $z$, the
same sequence of coefficients $c_j$, and with
the sequence $\{\lambda_j\}$ replaced by $\{\lambda_j+\epsilon_0\}$. Therefore, up to this point, the proof of the
theorem would work even with $F_z$ replaced by $F_{\epsilon_0,z}$. In the course of doing this, the angle $\theta_M$, however, changes to, say, $\theta_{M'}$. But since we have the liberty of choosing $\epsilon_0$, we choose it in such a way that
the angle $\theta_{M'}$ becomes an irrational multiple of $\pi$.
With this choice of $\epsilon_0$ and the analysis done so far, it is clear that the function $F_{\epsilon_0,z}(s)$
has infinitely many zeros on the critical line. Thus $F_z(s)$ also has infinitely many zeros on the
critical line when $\theta_{M}/\pi$ is rational. This completes the proof of Theorem 2 in all cases.

\section{Concluding remarks}\label{con}
In this work, we saw an application of the general theta transformation \eqref{thetaabz} and the integral \eqref{intnabla} equal to each of its expressions towards proving that a certain function involving the Riemann $\Xi$-function and the confluent hypergeometric function has infinitely many zeros on the critical line, thereby vastly generalizing Hardy's theorem. There are plethora of new modular-type transformations, that is, the transformations governed by the relation $a\to b$, where $ab=1$, that have linked to them certain definite integrals having $\Xi(t)$ under the sign of integration. See the survey article \cite{ingenious}, \cite{dixitmoll}, \cite{koshkernel}, and \cite[Section 15]{bdrz1}, for example. Recently in \cite[Theorems 1.3, 1.5]{dkmt}, a higher level theta transformation and the integral involving the Riemann $\Xi$-function linked to it was obtained. It may be interesting to see up to what extent one can extend Hardy's idea to obtain new interesting results. One thing is clear - when one has $\Xi^{2}(t)$ under the sign of integration, the sign change argument as in the case of $\Xi(t)$ cannot be used. Nevertheless, the higher level theta transformation in \cite[Theorems 1.3, 1.5]{dkmt} has two additional parameters in it, so it would be interesting to see what information could be extracted from it.

\begin{center}
\textbf{Acknowledgements}
\end{center}
The first author's research is supported by the SERB-DST grant RES/SERB/MA/\newline
P0213/1617/0021 whereas the third author is a SERB National Post Doctoral Fellow (NPDF) supported by the fellowship PDF/2017/000370. Both sincerely thank SERB-DST for the support. 

\end{document}